\documentclass[11pt]{amsart}

\usepackage{amsmath, amsfonts, amssymb, stmaryrd}

\newtheorem{theorem}{Theorem}

\newtheorem{proposition}[theorem]{Proposition}
\newtheorem{corollary}[theorem]{Corollary}

\theoremstyle{definition}

\theoremstyle{remark}
\newtheorem{remark}[theorem]{Remark}

\numberwithin{equation}{section}

\newcommand{\Prob}{\mathrm{Prob}}
\newcommand{\iDes}{\mathrm{iDes}}

\newcommand{\SEP}{\mathrm{SEP}}
\newcommand{\linf}{\ell_{\infty}}

\newcommand{\id}{\mathrm{id}}
\newcommand{\rev}{\mathrm{rev}}

\newcommand{\Sn}{\mathfrak{S}_n}
\newcommand{\sgn}{\mathrm{sgn}}

\begin{document}

\title{Riffle shuffles with biased cuts}

\author[Assaf]{Sami Assaf}
\address{Berkeley Quantitative, 140 Sherman St, Fairfield, CT 06824}
\email{sassaf@math.mit.edu}

\author[Diaconis]{Persi Diaconis}
\address{Department of Statistics, Stanford University, 390 Serra Mall, Stanford, CA 94305-4065}

\author[Soundararajan]{K. Soundararajan}
\address{Department of Mathematics, Stanford University, 450 Serra Mall, Building 380, Stanford, CA 94305-2125}
\email{ksound@math.stanford.edu}


\date{\today}



\begin{abstract}
  The well-known Gilbert-Shannon-Reeds model for riffle shuffles
  assumes that the cards are initially cut `about in half' and then
  riffled together. We analyze a natural variant where the initial cut
  is biased. Extending results of Fulman (1998), we show a sharp
  cutoff in separation and L-infinity distances. This analysis is
  possible due to the close connection between shuffling and
  quasisymmetric functions along with some complex analysis of a
  generating function.
\end{abstract}

\maketitle

\section{Introduction}
\label{sec:introduction}

We analyze a natural one-parameter model for riffle shuffling a deck
of $n$ cards. Roughly, the deck is cut into two piles with a binomial
$(n,\theta)$ distribution. Then the piles are riffled together
sequentially according to the following rule: if the left pile has $A$
cards and the right pile has $B$ cards, then drop the next card from
the bottom of the left pile with probability $A/(A+B)$. Continue until
all cards are dropped. Starting at the identity, let $P_{\theta}(w)$
be the probability of the permutation $w$ after one such
$\theta$-shuffle. Define convolution by 
\begin{equation}
  P_{\theta}^{*k}(w) = \sum_{v} P_{\theta}(v) P_{\theta}^{*(k-1)}(v^{-1}w),
\end{equation}
and define the uniform distribution by $U(w) = 1/n!$.

When $\theta = 1/2$, this is the widely studied Gilbert-Shannon-Reeds
model. The natural version with biased cuts was studied by
\cite{DFP1992}, \cite{Lalley1996,Lalley2000} and most thoroughly by
\cite{Fulman1998}. A literature review is in
Section~\ref{sec:shuffling} below. Here we study the rate of
convergence in the separation and $\linf$ metrics:
\begin{eqnarray}
  \SEP(k) & = & \max_{w} \left( 1 - \frac{P^{*k}(w)}{U(w)} \right) \\
  \linf(k) & = & \max_{w} \left| 1 - \frac{P^{*k}(w)}{U(w)} \right| .
  \label{e:metrics_linf}
\end{eqnarray}
Note that $\SEP(k)$ is bounded above by $1$, and $\linf(k)$ can be as
large as $n!-1$. Further, both $\SEP(k)$ and $\linf(k)$ are upper
bounds for the total variation metric:
\begin{displaymath}
  \| P^{*k} - U \|_{TV} = \frac{1}{2} \sum_{w} |P^{*k}(w) - U(w)| 
  \leq \SEP(k) \leq \linf(k).
\end{displaymath}

A main result of this note gives closed form expressions
\begin{eqnarray}
  \SEP(k) & = & 1 - \sum_{w \in \Sn} \sgn(w) \prod_{i=1}^{n} \left(
    \theta^i + (1-\theta)^i \right)^{kn_i(w)} 
  \label{e:sep_closed}\\
  \linf(k) & = & \sum_{w \in \Sn} \prod_{i=1}^{n} \left(
    \theta^i + (1-\theta)^i \right)^{kn_i(w)} - 1,
  \label{e:linf_closed}
\end{eqnarray}
where $n_i(w)$ is the number of $i$-cycles in the permutation
$w$. Using these formulae we prove the following.

\begin{theorem}
  For the $\theta$-biased riffle shuffle measure on $\Sn$, let 
  \begin{equation}
    k = \Big \lfloor\frac{2\log n -\log 2 + c}{-\log (\theta^2
      +(1-\theta)^2)} \Big\rfloor.
    \label{e:k}
  \end{equation}
  Then
  \begin{eqnarray}
    \SEP(k) & \sim & \exp(e^{-c})-1 \\
    \linf(k) & \sim & 1-\exp(-e^{-c})
  \end{eqnarray}
  for any fixed real $c$ as $n$ tends to $\infty$. Here $0 < \theta <
  1$ is fixed.
  \label{thm:main}
\end{theorem}

An upper bound on separation of this form is given in
\cite{Fulman1998}. Theorem~\ref{thm:main} shows this bound is tight,
holds also for $\linf$, and establishes the cutoff phenomenon. Note
that, as a function of $\theta$, $k$ as defined in \eqref{e:k} above
is smallest when $\theta = 1/2$, so unbiased cuts lead to fastest
mixing.

Background on Markov chains and shuffling is given in
Section~\ref{sec:shuffling}. There is an intimate connection between
these biased shuffles and quasisymmetric functions explained in
Section~\ref{sec:quasisymmetric} where we prove \eqref{e:sep_closed}
and \eqref{e:linf_closed}. The upper bound in \cite{Fulman1998} is
derived using a strong stationary time. This is shown to be exact and
equivalent to \eqref{e:sep_closed} in
Section~\ref{sec:stationary}. The proof of Theorem~\ref{thm:main},
which has extensions to allow $\theta$ to depend on $n$ (e.g. $\theta
= 1/n$), is in Section~\ref{sec:main}.

\section{Riffle Shuffling}
\label{sec:shuffling}

A superb introduction to Markov chains which treats riffle shuffling
and stationary times is the book by \cite{LPW2009}. The analysis of
riffle shuffling has connections to algebra, geometry and
combinatorics; a detailed survey is in \cite{Diaconis2003}. The
results and references in \cite{ADS2011} and \cite{CH2010} bring this
up to date.

For present purposes, the following extension is needed. Let $1 \leq a
\leq \infty$, and let $\boldsymbol{\theta} = (\theta_1, \theta_2,
\ldots, \theta_a)$, with $0 \leq \theta_i \leq 1$ and $\theta_1 +
\cdots + \theta_a = 1$, be fixed. A $\boldsymbol{\theta}$-shuffle of a
deck of $n$ cards proceeds as follows: Choose $\{N_i\}_{i=1}^{a}$ from
the multinomial$(n,\boldsymbol{\theta})$ distribution, that is, with
the distribution of $n$ balls being dropped into $a$ boxes
independently according to $\boldsymbol{\theta}$. Cut the deck into
$a$ packets of sizes $N_1,N_2, \ldots, N_a$ (some of the packets may
be empty). Now sequentially drop cards from the bottom of each packet,
choosing to drop from pack $i$ with probability proportional to its
current packet size. Continue until all cards have been dropped into a
single pile. Let $P_{\boldsymbol{\theta}}$ denote the associated
measure on $\Sn$. Note that several more detailed descriptions of
$P_{\boldsymbol{\theta}}$ appear in \cite{Fulman1998}.

When $a=2$ and $\theta_1 = \theta_2 = 1/2$, this is the basic
Gilbert-Shannon-Reeds measure. When $a = 2$ and $\theta_1 = \theta$,
$\theta_2 = 1-\theta$, this is the $\theta$-biased shuffle studied in
the present paper. The measures $P_{\boldsymbol{\theta}}$ were studied
by \cite{DFP1992} who prove that they convolve nicely: if
$\boldsymbol{\theta} = (\theta_1,\ldots,\theta_a)$ and
$\boldsymbol{\eta} = (\eta_1, \ldots, \eta_b)$, then set
$\boldsymbol{\theta} * \boldsymbol{\eta} = (\theta_1\eta_1, \ldots,
\theta_1\eta_b, \theta_2\eta_1, \ldots, \theta_a\eta_b)$, a vector of
length $ab$.

\begin{proposition}[\cite{DFP1992}]
  On $\Sn$, we have
  \begin{displaymath}
    P_{\boldsymbol{\theta}} * P_{\boldsymbol{\eta}} = P_{\boldsymbol{\theta}*\boldsymbol{\eta}}.
  \end{displaymath}
  \label{prop:convolve}
\end{proposition}

Thus $P_{\boldsymbol{\theta}}^{*k} = P_{\boldsymbol{\theta}^{*k}}$,
and the combinatorics of $P_{\boldsymbol{\theta}}$ determines the
convolution powers. \cite{Fulman1998} works out many properties of
these measures giving closed formulae and asymptotics for the
distribution of cycle structure, inversions and descents.

When $\theta = 1/2$, a sharp analysis of the rate of convergence for
the Gilbert-Shannon-Reeds measure in total variation distance appears
in \cite{BD1992}. It is an open problem to give a similarly sharp
analysis for the measures $P_{\boldsymbol{\theta}}^{*k}$.

\subsection*{Hyperplane walks}

Our $\theta$-shuffles may be studied from other points of view as
well. They are a special case of hyperplane walks introduced in
\cite{BHR1999} and further studied in \cite{BD1998} and more recently
in \cite{AD2010} and \cite{DPR}. Further, they fall into the class of
``Hopf-square'' walks studied in \cite{DPR}. Each of these
perspectives adds to our picture. A brief commentary follows.

The braid arrangement is based on the $\binom{n}{2}$ hyperplanes
$H_{i,j} = \{x \in \mathbb{R}^n \ | \ x_i = x_j\}$, $1\leq i<j \leq
n$. This divides $\mathbb{R}^n$ into chambers and faces. As shown in
\cite{BHR1999}, the chambers are indexed by permutations and the faces
are indexed by block ordered set partitions. There is a simple
projection operator which, given a chamber $C$ and a face $F$, returns
the chamber $C * F$ that is adjacent to $F$ and closest to $C$ (in the
sense of crossing the fewest number of hyperplanes). Details are in
\cite{BHR1999,BD1998}. It is shown there that projection operates as a
kind of inverse riffle shuffle. Put a probability measure on faces of
form $S,S^{c}$, with $S \subset [n]$, giving probability
$\theta^{|S|}(1-\theta)^{n-|S|}$ to each ($S$ may be empty). The
resulting hyperplane walk may be explained as follows: Picture a deck
of $n$ cards in order. For each card, flip an independent
$\theta$-coin. Remove all cards where the coin comes up heads, keeping
their relative order fixed, and move them to the top of the deck. This
is precisely an inverse $\theta$-shuffle.

The theory of \cite{BHR1999,BD1998} gives useful expressions for the
eigenvalues of any hyperplane walk. Specialized to $\theta$-shuffles,
they show there is one eigenvalue $\beta_w$ for each permutation $w
\in \Sn$. Further, \cite{DPR} gives a description of the left eigen
vectors. These give right eigen vectors and values of the ``forward''
$\theta$-shuffles.

As one example, \cite{BD1998} gives a rate of convergence after
$k$-steps. In the present case, this reads
\begin{equation}
  \| K^{k} - U \|_{TV} \leq \sum_{1 \leq i < j \leq n} \beta_{i,j}^{k}
\end{equation}
with $\beta_{i,j} = \sum_{F \subseteq H_{i,j}} w(F)$. By symmetry,
$\beta_{i,j} = \beta_{1,2}$ is constant in $i,j$. The sum is over all
set partitions $S,S^{c}$ where either $\{1,2\} \subseteq S$ or
$\{1,2\} \subseteq S^{c}$. So $\{1,2\}$ contributes $\sum_{A \subseteq
  [n-1]} \theta^{2} \theta^{|A|} (1-\theta)^{n-2-|A|} = p^2$, the
compliment contributes $(1-\theta)^2$, and so $\beta_{i,j} = \theta^2
+ (1-\theta)^2$. The bound above becomes
\begin{equation}
  \| K^{k} - U \|_{TV} \leq \binom{n}{2} \left( \theta^2 +
    (1-\theta)^2 \right)^{k}.
\end{equation}
This is exactly the birthday bound derived differently below. Of
course, these are just upper bounds, and it is of interest to know if
they can be improved. The theory developed below shows that
\begin{equation}
  \| K^{k} - U \|_{TV} \leq \SEP(k) \leq \binom{n}{2} \left( \theta^2
    + (1-\theta)^2 \right)^{k}.
\label{e:bound}
\end{equation}
for fixed $\theta$ in $(0,1)$. Theorem~\ref{thm:main} shows that
$\SEP(k) \sim \binom{n}{2} \left( \theta^2 + (1-\theta)^2
\right)^{k}$, so the bound is best possible.

Recall that any $w \in \Sn$ has a unique factorization as a product of
decreasing Lyndon words: $w = \ell_1 \ell_2 \cdots \ell_k$. Here
$\ell_i$ is Lyndon if it is lexicographically least among all cyclic
rearrangements (so $132$ is Lyndon but $213$ is not). For example
$236415 = 236 \cdot 4 \cdot 15$. The theorem in \cite{DPR} shows
\begin{equation}
  \beta_w = \prod_{i=1}^{k} \left( \theta^{|\ell_i|} +
    (1-\theta)^{|\ell_i|} \right) 
\end{equation}
where $|\ell_i|$ is the length of the Lyndon word $\ell_i$. If $w$ is
the reverse of the identity, then all $|\ell_i|=1$ and $\beta_{w}=1$. The
second eigenvalue is $\theta^2 + (1-\theta)^2$ with multiplicity
$\binom{n}{2}$, so the bound \eqref{e:bound} uses precisely these
eigen values. More generally, the eigen values are $\prod_{i=1}^{n}
\left(\theta^i + (1-\theta)^i \right)^{a_i}$ for any $0 \leq a_i \leq n$
with $\sum i a_i = n$, each with multiplicity $n!/(\prod_i i^{a_i}
a_i!)$.

\section{Quasisymmetric Functions}
\label{sec:quasisymmetric}

Background on symmetric function theory is in \cite{Macdonald1995}
with \cite{ECII} developing the extension to quasisymmetric
functions. We work with infinitely many variables $X =
\{x_i\}_{i=1}^{\infty}$. The space of quasisymmetric functions
homogeneous of degree $n$ has dimension $2^{n-1}$. A basis for this
space is indexed by subsets of $[n-1] = \{1,2,\ldots,n-1\}$ or,
equivalently, by compositions of $n$. We use the following bijection
between subsets $D = \{D_1 < D_2 < \cdots < D_{a-1}\}$ of $[n-1]$ and
compositions $\alpha = (\alpha_1,\alpha_2,\ldots,\alpha_a)$ of $n$ to
identify subsets and compositions, which we denote by $\alpha
\leftrightarrow D(\alpha)$:
\begin{eqnarray*}
  (\alpha_1, \alpha_2, \ldots, \alpha_{a}) 
  & \longmapsto &
  \{\alpha_1,\alpha_1+\alpha_2,\ldots,\alpha_1+\cdots+\alpha_{a-1}\},
  \\
  (D_1, D_2-D_1, \ldots, n-D_{a-1})
  & \longmapsfrom &
  \{D_1 < D_2 < \ldots < D_{a-1}\}.
\end{eqnarray*}

The \emph{monomial} quasisymmetric function basis is defined by
\begin{equation}
  M_{\alpha}(X) = \sum_{i_1 < i_2 < \cdots < i_{a}}
  x_{i_1}^{\alpha_1} x_{i_2}^{\alpha_2} \cdots x_{i_a}^{\alpha_a}.
  \label{e:monomial}
\end{equation}
For example, $M_{(1,2,1)}(X) = \sum_{i_1 < i_2 < i_3} x_{i_1}
x_{i_2}^{2} x_{i_3}$. 

The \emph{fundamental} quasisymmetric function basis of
\cite{Gessel1984} is defined by
\begin{equation}
  Q_{D}(X) = \sum_{\substack{i_1 \leq \cdots \leq i_n \\ i_j = i_{j+1} \Rightarrow j \not\in D}} 
  x_{i_1} \cdots x_{i_n} .
  \label{e:fundamental}
\end{equation}
For example, for $n=4$, $Q_{\{1\}}(X) = \sum_{i_1 < i_2 \leq i_3 \leq
  i_4} x_{i_1} x_{i_2} x_{i_3} x_{i_4}$. Expressed in terms of
monomial quasisymmetric functions, $Q_{\{1\}}(X) = M_{(1,3)}(X) +
M_{(1,2,1)}(X) + M_{(1,1,2)}(X) + M_{(1,1,1,1)}(X)$. In general, the
fundamental basis is related to the monomial basis by
\begin{equation}
  Q_{D(\beta)}(X) = \sum_{\alpha \ \mathrm{refines} \ \beta} M_{\alpha}(X),
  \label{e:F2M}
\end{equation}
where a composition $\alpha$ of length $a$ refines the composition
$\beta$ of length $b$ if there exist indices $0=i_0, i_1, i_2, \ldots,
i_b=a$ such that $\alpha_{i_{j-1}+1} + \cdots + \alpha_{i_j} =
\beta_{j}$. For example, both $(1,2,1)$ and $(1,1,2)$ refine $(1,3)$
but $(2,1,1)$ does not.

\cite{Stanley2001}, based on results in \cite{Fulman1998}, established
a sharp connection between $\boldsymbol{\theta}$-shuffling and
quasisymmetric functions.

\begin{theorem}[\cite{Stanley2001}(Theorem 2.1)] Let $w \in \Sn$ and
  $\boldsymbol{\theta} = (\theta_1,\theta_2,\ldots,\theta_a)$ be given. Then
  \begin{displaymath}
    P_{\boldsymbol{\theta}}(w) = Q_{\iDes(w)}(\boldsymbol{\theta}),
  \end{displaymath}
  where $\iDes(w) = \mathrm{Des}(w^{-1})$ is the inverse descent set
  of $w$.
  \label{thm:stanley}
\end{theorem}

This identification together with \eqref{e:F2M} gives a useful
inequality which shows that separation and $\linf$ are achieved at the
reversal and the identity permutations, respectively.

\begin{proposition}
  For permutations $w$ and $u$, if $\iDes(w)$ contains $\iDes(u)$,
  then $\Prob(w) \leq \Prob(u)$ with equality if and only if $\iDes(w)
  = \iDes(u)$.
  \label{prop:refine}
\end{proposition}

\begin{proof}
  First note that $\alpha$ refines $\beta$ if and only if $D(\alpha)$
  contains $D(\beta)$. Let $\alpha$ and $\beta$ be such that
  $D(\alpha) = \iDes(w)$ and $D(\beta) = \iDes(u)$. From \eqref{e:F2M}
  and the transitivity of refinement, we have
  \begin{eqnarray*}
    Q_{D(\beta)}(X) & = & \sum_{\gamma \ \mathrm{refines} \ \beta}
    M_{\gamma}(X) \\
    & = & \sum_{\gamma \ \mathrm{refines} \ \alpha}
    M_{\gamma}(X) + \sum_{\substack{\gamma\prime \ \mathrm{refines} \
        \beta \\ \gamma\prime \ \mathrm{not} \ \mathrm{refine} \
        \alpha} }
    M_{\gamma\prime}(X) \\
    & = & Q_{D(\alpha)}(X) + \sum_{\substack{\gamma\prime \
        \mathrm{refines} \ \beta \\ \gamma\prime \ \mathrm{not} \
        \mathrm{refine} \ \alpha} } M_{\gamma\prime}(X).
  \end{eqnarray*}
  Furthermore, $\alpha \neq \beta$ if and only if $\beta$ does not
  refine $\alpha$, in which case the summand contains the term
  $M_{\beta}$. Since the $x_i$ are probabilities, they are all
  nonnegative, thus making $Q_{D(\alpha)}$ strictly less than
  $Q_{D(\beta)}$.
\end{proof}

In the partial order on subsets or, equivalently, composition, $[n-1]
= D(1^n)$ is the unique minimal element and $\varnothing = D(n)$ is
the unique maximal element. Therefore Proposition~\ref{prop:refine}
has the following consequence.

\begin{corollary}
  For any $\boldsymbol{\theta}$, we have
  \begin{eqnarray*}
    \SEP(P_{\boldsymbol{\theta}}) & = & 1 - n! \cdot
    Q_{[n-1]}(\boldsymbol{\theta}) \\
    \linf(P_{\boldsymbol{\theta}}) & = & \max(1 - n! \cdot
    Q_{[n-1]}(\boldsymbol{\theta}), n! \cdot Q_{\varnothing}(\boldsymbol{\theta}) - 1).
  \end{eqnarray*}
  \label{cor:metrics}
\end{corollary}

\begin{remark}
  When $\boldsymbol{\theta} = (\theta, 1-\theta)^{*k}$, we show below
  that the maximum is taken on at the second argument. This is not
  always the case. On the cyclic group $\mathcal{C}_3$, with $\mu(1) =
  \mu(-1) = \frac{1}{2}, \mu(0) = 0$, we have $3\mu(1)-1 =
  \frac{1}{2}$ and $1-3\mu(0)=1$.
\end{remark}

For some permutations, the associated quasisymmetric functions are
easy to write down. This happens in particular if the quasisymmetric
function is symmetric. Below we need the elementary symmetric
functions $e_n(X)$, the complete homogeneous symmetric functions
$h_n(X)$, and the power sum symmetric functions $p_n(X)$. For $\lambda$
a partition with $n_i = n_i(\lambda)$ parts equal to $i$, define
$e_{\lambda} = \prod_i e_i^{n_i}$, $h_{\lambda} = \prod_i h_i^{n_i}$,
$p_{\lambda} = \prod_i p_i^{n_i}$. As $\lambda$ ranges over partitions
of $n$, these are the familiar bases for the homogeneous symmetric
functions of degree $n$.

Note that
\begin{equation}
  e_n(X) = Q_{[n-1]}(X), \hspace{2em}
  h_n(X) = Q_{\varnothing}(X), \hspace{2em}
  p_n(X^{*k}) = \left( p_n(X) \right)^k.
  \label{e:bases}
\end{equation}

\begin{theorem}
  For any $\boldsymbol{\theta}$, with $\id = 1,2,\ldots,n$ and $\rev =
  n,n-1,\ldots,1$, we have
  \begin{eqnarray}
    P_{\boldsymbol{\theta}}^{*k}(\rev) & = & \sum_{\lambda \vdash n}
    (-1)^{n-\ell(\lambda)} z_{\lambda}^{-1}
    \prod_{i=1}^{n} p_i(\boldsymbol{\theta})^{kn_i(\lambda)},
    \label{e:P_rev}\\
    P_{\boldsymbol{\theta}}^{*k}(\id) & = & \sum_{\lambda \vdash n}
    z_{\lambda}^{-1} \prod_{i=1}^{n}
    p_i(\boldsymbol{\theta})^{kn_i(\lambda)}, 
    \label{e:P_id}
  \end{eqnarray}
  where $\ell(\lambda)$ is the number of parts of $\lambda$ and
  $z_{\lambda} = \prod_i i^{n_i(\lambda)} n_i(\lambda)!$.
  \label{thm:SEP}
\end{theorem}

\begin{proof}
  The result follows from Theorem~\ref{thm:stanley}, \eqref{e:bases}
  and the standard expansions \cite{Macdonald1995}
  \begin{equation}
    e_n = \sum_{\lambda} \epsilon_{\lambda} z_{\lambda}^{-1}
    p_{\lambda} \hspace{2em} \mbox{and} \hspace{2em}
    h_n = \sum_{\lambda} z_{\lambda}^{-1} p_{\lambda}.
    \label{e:power}
  \end{equation}
\end{proof}

\begin{remark}
  For both \eqref{e:P_rev} and \eqref{e:P_id} in the theorem, when
  $\lambda = (1^n)$, $z_{\lambda}^{-1} = 1/n!$ and $\prod_i
  p_i(\boldsymbol{\theta})^{k n_i(\lambda)} = 1$. Thus the lead term is
  $1/n!$ and all other terms are strictly less than $1$. As $k$ tends
  to $\infty$, these terms tend to $0$ and
  $P_{\boldsymbol{\theta}}^{*k}(\id) \sim P_{\boldsymbol{\theta}}^{*k}(\rev)
  \sim \frac{1}{n!}$. Of course, our work is to quantify this
  convergence.
\end{remark}

\begin{corollary}
  For any $\boldsymbol{\theta}$ and all $k \geq 0$, we have
  \begin{eqnarray*}
    \SEP(P_{\boldsymbol{\theta}}^{*k}) & = & 1 - n!
    P_{\boldsymbol{\theta}}^{*k}(\rev), \\
    \linf(P_{\boldsymbol{\theta}}^{*k}) & = &
    n!P_{\boldsymbol{\theta}}^{*k}(\id) - 1.
  \end{eqnarray*}
  \label{cor:extremes}
\end{corollary}

\begin{proof}
  The first equality follows from the definition. For the second
  inequality,
  \begin{displaymath}
    \linf(P_{\boldsymbol{\theta}}^{*k}) =
    \max(1-n!P_{\boldsymbol{\theta}}^{*k}(\rev),
    n!P_{\boldsymbol{\theta}}^{*k}(\id)-1).
  \end{displaymath}
  In comparing terms, the $1$ cancels in both, and the second term is
  a sum of positive terms while the first has the same terms, some
  with negative signs.
\end{proof}

Specializing to $\theta$-biased shuffles, Corollary~\ref{cor:extremes}
and Theorem~\ref{thm:SEP} imply \eqref{e:sep_closed} and
\eqref{e:linf_closed}.

\section{Strong Stationary Times}
\label{sec:stationary}

Repeated shuffling from any of the measures in
Section~\ref{sec:shuffling} forms a Markov chain $id = W_0, W_1, W_2,
\ldots$ taking values in $\Sn$. A \emph{strong stationary time} (SST)
$T$ is a stopping time (meaning $P\{T > k\}$ only depends on $W_0,
W_1, \ldots, W_k$) such that for all $k \geq 0$ and all $w \in \Sn$,
\begin{equation}
  P \{ W_k = w \ | \ T \leq k\} = U(w).
  \label{e:SST}
\end{equation}
We will build an SST for the Markov chain induced by
$P_{\boldsymbol{\theta}}^{*k}$. A basic proposition of this theory
\cite{LPW2009}[Lemma 6.1] is
\begin{equation}
  \SEP(k) \leq P \{ T > k\} 
  \hspace{1em}
  \mbox{for all $k \geq 0$.}
  \label{e:SEP_SST}
\end{equation}
Further, \cite{AldousDiaconis1987} shows that there always exists a
fastest SST $T^*$ satisfying \eqref{e:SEP_SST} with equality for all
$k$.

Background on stationary times is in \cite{DF1990}. In this section,
we build a fastest SST (following \cite{AldousDiaconis1987} and
\cite{Fulman1998}) involving a birthday problem to bound the right
hand side of \eqref{e:SEP_SST}. Solving this birthday problem by
inclusion-exclusion gives a probabilistic proof of
\eqref{e:sep_closed}, Theorem~\ref{thm:SEP} and even the expression
for the elementary symmetric function $e_n$ in terms of the power sums
\eqref{e:power}.

\subsection*{Constructing an SST for $P_{\boldsymbol{\theta}}^{*k}$}

Consider the inverse process in which cards are labeled $i$ with
probability $\theta_i$ independently. Then all the cards labeled $1$
are removed, keeping them in their same relative order, followed by
all cards labeled $2$, and so on. This is one inverse
$\boldsymbol{\theta}$-shuffle. Repetitions may be realized by labeling
each card with a vector with coordinates chosen independently from
$\boldsymbol{\theta}$. The first shuffle is read off the first
coordinate of each card, the second shuffle off the second coordinate,
and so on. Conceptually, each card may be labeled with a vector of
infinite length.

Consider the first time $T$ that the first $T$ coordinates of the $n$
cards are distinct. Repeated inverse shuffling sorts the vectors
lexicographically, leaving the card with the smallest vector on top,
the next smallest second, and so on. By symmetry, at time $T$, the
deck is uniformly distributed, even conditional on $T=k$. This is
\eqref{e:SST}. Further, this $T$ is fastest. To see this, note that
the reversal permutation $\rev = n,n-1,\ldots,1$ is a \emph{halting
  state}: $P\{T \leq k\} \leq P\{W_T = \rev\}$. Indeed, if $W_T =
\rev$, then every pair of cards must have a distinct label. Existence
of a halting state implies that $T$ is fastest (\cite{DF1990}[Remark
2.39] and \cite{LPW2009}[Remark 6.12]), separation is achieved at
$\rev$, and
\begin{equation}
  \SEP(k) = P\{T > k\}
  \hspace{1em}
  \mbox{for all $k \geq 0$.}
  \label{e:SEP_SST2}
\end{equation}
To work with the right hand side of \eqref{e:SEP_SST2}, let $A_{i,j}$
be the event that the first $k$ coordinates of the labels on the cards
$i$ and $j$ are equal. Thus $P\{A_{i,j}\} = \left(\sum_a \theta_a^2
\right)^{k}$ and
\begin{equation}
  \{ T > k\} = \bigcup_{1 \leq i < j \leq n} \{A_{i,j}\}.
  \label{e:Aij}
\end{equation}
Bounding the probability of the union by the sum of the probabilities
yields
\begin{equation}
  \SEP(k) \leq \binom{n}{2} \left(\sum_a \theta_a^2\right)^{k}.
\end{equation}
This bound is also derived in \cite{Fulman1998}. The asymptotics of
Section~\ref{sec:main} show it is quite accurate.

\subsection*{Inclusion-Exclusion and the Birthday Problem}

Consider this version of the birthday problem: $n$ balls are dropped
independently into $B$ boxes with the chance of box $i$ being
$\eta_i$. If $B_{i,j}$ is the event that balls $i$ and $j$ both wind
up in the same box, the chance of success (having two or more balls in
the same box) is
\begin{equation}
  P(\mathrm{success}) = P \left( \bigcup_{1 \leq i < j \leq n} B_{i,j}
  \right).
  \label{e:union}
\end{equation}
Elementary considerations show that the chance of failure (all balls
in distinct boxes) is expressible using elementary symmetric functions
$e_n$ as $1-P(\mathrm{success}) = n!e_n(\eta_1,\ldots,\eta_B)$. Using
the expression for $e_n$ in terms of the power sums \eqref{e:power}
gives
\begin{equation}
  P(\mathrm{success}) = 1 - \sum_{w \in \Sn} \sgn(w)
  p_{\lambda(w)}(\boldsymbol{\eta}) = 1 - n! \sum_{\lambda
    \vdash n} (-1)^{n-\ell(\lambda)} z_{\lambda}^{-1}
  p_{\lambda}(\boldsymbol{\eta}) 
  \label{e:success}
\end{equation}

The inclusion-exclusion expansion of \eqref{e:union} gives a sum of
polynomials which must match the neat expressions in
\eqref{e:success}. This may be seen explicitly using the
inclusion-exclusion formula for the chromatic polynomial in
\cite{Stanley1995}. For example,
\begin{displaymath}
  P\{B_{1,2} \cup B_{1,3} \cup B_{2,3}\} = 3P(B_{1,2}) -
  3P(B_{1,2} \cap B_{2,3}) + P(B_{1,2} \cap B_{1,3} \cap B_{2,3}) =
  3(\sum p_j^2) - 2(\sum p_j^3),
\end{displaymath}
while \eqref{e:success} gives $6( -\frac{1}{2}
p_{(2,1)}(\boldsymbol{\eta}) + \frac{1}{3} p_3(\boldsymbol{\eta}) )$
matching \eqref{e:union}.

\begin{remark}
  Since separation is achieved (uniquely) at the reversal permutation,
  \eqref{e:SEP_SST2}, \eqref{e:Aij}, \eqref{e:union},
  \eqref{e:success} give a probabilistic proof of
  Theorem~\ref{thm:SEP}.
\end{remark}

\begin{remark}
  This connection between inclusion--exclusion, birthday problems and
  symmetric functions seems generally useful. See, for example,
  \cite{MS2004}[pg. 604--605].
\end{remark}

\section{Main Result}
\label{sec:main}

This section derives the asymptotic results of Theorem~\ref{thm:main}
and some extensions. Without loss of generality, suppose $1/2 \le
\theta \le 1$. To bound the $\linf$ distance, using
Corollary~\ref{cor:extremes} together with \eqref{e:linf_closed} and
\eqref{e:sep_closed}, we are interested in
\begin{equation}
  \ell(k, n) =\sum_{w \in S_n} \prod_{j} \theta_j^{kn_j(w)}, 
\end{equation}
where $\theta_j = \theta^j + (1-\theta)^j$ and $n_j(w)$ denotes
the number of $j$ cycles in the permutation $w$.  If
$$ 
f_n(x_1,\ldots, x_n)  = \sum_{w\in S_n} \prod_{j} x_j^{n_j(w)} 
$$ 
then we have the identity 
\begin{equation}
  \sum_{n=0}^{\infty}\frac{ z^n}{n!}  f_n(x_1,\ldots,x_n)   
  = \exp\Big( \sum_{j=1}^{\infty} \frac{z^j}{j} x_j \Big). 
\end{equation}
Therefore we have that
\begin{equation}
  \sum_{n=0}^{\infty} \frac{z^n}{n!} \ell(k,n) = \exp\left(
    \sum_{j=1}^{\infty} \frac{z^j }{j} \theta_j^k \right).
\end{equation}

\begin{theorem}
  Define
  $$ 
  M= M(k,n) = \sum_{j=2}^{\infty} n^j \theta_j^k. 
  $$ 
  If $M \le \sqrt{n}/(10\log n)$, then we have 
  $$ 
  { \ell(k,n)}= \exp\left( \sum_{j=2}^{\infty} \frac{n^j}{j} \theta_j^k 
  \right) \left( 1+ O\left( \frac{1+M}{\sqrt{n}}\right) \right).
  $$ 
  \label{thm:sound}
\end{theorem}

\begin{proof}
  Set $F_k(z) = \sum_{j=1}^{\infty} \frac{z^j }{j} \theta_j^k$.  By
  the residue theorem we have
  $$ 
  \ell(k,n) = \frac{n!}{2\pi i} \int_{|z|=n} \exp(F_k(z)) z^{-n} \frac{dz}{z} 
  = \frac{n!}{2\pi n^n} \int_{-\pi}^{\pi} \exp(F_k(ne^{ix}) - inx) dx. 
  $$ 
  We divide the integral into the ranges when $|x|\le (\log n)/\sqrt{n}$ 
  which gives the main contribution, and $\pi \ge |x|> (\log n)/\sqrt{n}$.  

  Consider first the range $|x|\le (\log n)/\sqrt{n}$.  Here we have
  $$ 
  F_k(n e^{ix}) = ne^{ix} + \sum_{j=2}^{\infty} \frac{n^j}{j} \theta_j^k
  e^{ijx} = ne^{ix} + \sum_{j=2}^{\infty} \frac{n^j}{j} \theta_j^k +
  O(|x| M),
  $$ 
  since $e^{ijx} = 1+ O(j|x|)$.  Therefore, using $ne^{ix} = n +inx -n
  x^2/2 + O(|x|^3n)$ and Stirling's formula, the integral over this
  region is
  \begin{displaymath}
    \frac{n!}{2\pi n^n} \int_{|x|\le (\log n)/\sqrt{n}} \exp\Big( n -
    \frac{nx^2}{2} +\sum_{j=2}^{\infty} \frac{n^j}{j} \theta_j^k +
    O\Big(|x|^3 n+ |x| M\Big)\Big) dx 
  \end{displaymath}
  which reduces to
  \begin{equation}
    \exp\Big(\sum_{j=2}^{\infty}
    \frac{n^j}{j} \theta_j^k\Big) \Big( 1 +
    O\Big(\frac{1+M}{\sqrt{n}}\Big)\Big).
  \end{equation}

  Now consider the range $\pi \ge |x|> (\log n)/\sqrt{n}$.  Here we
  have
  $$ 
  Re (F_k(ne^{ix})) \le  F_k(n)  - n (1-\cos(x)) \le F_k(n) - c(\log n)^2, 
  $$ 
  for some positive constant $c$. Using Stirling's formula again, the
  contribution of this segment of the integral is therefore
  $$ 
  \ll \frac{n!}{n^n} \exp(F_k(n) - c(\log n)^2) \ll \sqrt{n}
  \exp\Big(\sum_{j=2}^{\infty} \frac{n^j}{j} \theta_j^k - c(\log
  n)^2\Big),
  $$ 
  which may be absorbed into our error term.  
\end{proof}

From this Theorem we can read off the behavior of the $\ell^{\infty}$
distance after $k$ biased shuffles. First consider the case when
$(1-\theta)\log n$ is large. In this range put
\begin{equation}
  k = \Big \lfloor\frac{2\log n -\log 2 + c}{-\log (\theta^2
    +(1-\theta)^2)} \Big\rfloor.
\end{equation}
We find that the contribution to $M(k,n)$ arises mainly from $j=2$ and
so $M(k,n) \ll e^{-c}$, and we have
\begin{equation} 
  \ell(k,n) \sim \exp(e^{-c}),
\end{equation}
so that the $\ell^{\infty}$ distance behaves like $\exp(e^{-c})-1$,
and similarly the separation distance behaves like $1- \exp(-e^{-c})$,
in agreement with \cite{DFP1992}.

Next consider the case when $(1-\theta) \log n = \kappa \in
[0,\infty)$. Keep the notation above for $k$, here we find that $n^2
\theta_2^k = 2 e^{-c}$, as before, and for $j\ge 3$,
\begin{equation}
  n^j \theta_j^k \sim \exp( \frac j2 (-\kappa + \log 2 -c)\Big).
\end{equation}
Therefore, if $c> \log 2 - \kappa$, then $M(k,n)$ is small, and
Theorem~\ref{thm:sound} applies.  Moreover in this case we have
\begin{equation}
  \ell (k,n) \sim \exp\Big( e^{-c} + \sum_{j=3}^{\infty} \frac{1}{j}
  \exp(\frac j2 (-\kappa+\log 2 -c))\Big).
\end{equation}

Finally, consider the extreme case $\theta= 1-1/n$.  It is convenient
here to define $k = n\log n +cn$.  Then $n^j \theta_j^k \sim e^{-jc}$
for $j\ge 2$, and $M(k,n)$ is small provided $c>0$.  In that case we
have
\begin{equation}
  \ell(k,n) \sim \exp\Big(\sum_{j=2}^{\infty} \frac{e^{-jc}}{j} \Big) =
  \frac{e^{-e^{-c}}}{1-e^{-c}}.
\end{equation}
Compare with Theorem 1.1 of \cite{DFP1992}.

\section*{Acknowledgements}
The authors thank Amy Pang for helpful conversations
about the hyperplane perspective, and Jason Fulman for careful
comments and references.

\bibliographystyle{alpha}
\bibliography{biased}

\end{document}